\documentclass[11pt]{article}
\usepackage{graphicx}
\usepackage{tikz}
\usepackage{amssymb}
\usepackage{amsmath,amsthm}
\usepackage{verbatim}
\usepackage{amsfonts}
\usepackage[T1]{fontenc}
\usepackage{multicol}
\usepackage{color}

\newtheorem{thm}{Theorem}
\newtheorem{prop}{Proposition}[section]

\newtheorem{defn}[thm]{Definition}
\newtheorem{obs}[prop]{Observation}
\newtheorem{lem}[thm]{Lemma}

\newcommand{\bF}{\mathbb{F}}
\newcommand{\bN}{\mathbb{N}}
\newcommand{\bfa}{\mathbf{a}}
\newcommand{\bfb}{\mathbf{b}}

\newcommand{\mfa}{\mathfrak{A}}
\newcommand{\mfi}{\mathfrak{i}}
\newcommand{\mfj}{\mathfrak{j}}
\newcommand{\mcS}{\mathcal{S}}
\newcommand{\bft}{\mathbf{t}}
\newcommand{\bfx}{\mathbf{x}}
\newcommand{\bfy}{\mathbf{y}}

\newcommand{\X}{\mathcal{X}}
\newcommand{\Y}{\mathcal{Y}}

\newcommand{\fD}{f^{(D)}}
\newcommand{\Z}{\mathbb{Z}}

\title{The Weighted Davenport constant of a group and a related extremal problem - II }
\author{Niranjan Balachandran\footnote{Department of Mathematics, IIT Bombay. Email: niranj (at) math.iitb.ac.in}\ \  and 
Eshita Mazumdar\footnote{Stat-Math Unit, ISI Bengaluru. Email: eshita\_vs (at) isibang.ac.in}}

\begin{document}
\maketitle
\begin{abstract} 
For a finite abelian group $G$ with $\exp(G)=n$  and an integer $k\ge 2$, Balachandran and Mazumdar \cite{BM} introduced the extremal function $\fD_G(k)$ which is defined to be $\min\{|A|: \emptyset \neq A\subseteq[1,n-1]\textrm{\ with\ }D_A(G)\le k\}$ (and $\infty$ if there is no such $A$), where $D_A(G)$ denotes the $A$-weighted Davenport constant of the group $G$. Denoting $\fD_G(k)$ by $\fD(p,k)$ when $G=\bF_p$ (for $p$ prime), it is known (\cite{BM}) that $p^{1/k}-1\le \fD(p,k)\le O_k(p\log p)^{1/k}$ holds for each $k\ge 2$ and $p$ sufficiently large, and that for $k=2,4$, we have the sharper bound $\fD(p,k)\le O(p^{1/k})$. It was furthermore conjectured that $\fD(p,k)=\Theta(p^{1/k})$. In this short paper we prove that $\fD(p,k)\le 4^{k^2}p^{1/k}$ for sufficiently large primes $p$.
\end{abstract}


2010 AMS Classification Code:  11B50, 11B75, 05D40.

\section{Introduction}
For a prime $p$, and  $a\ne b, a,b\in\bF_p$, we shall borrow the notation $[a,b]$ (from the usual integer case) to denote the set $\{a,a+1,\ldots,b\}$ for $a\neq b\in\bF_p$. Throughout this paper,  we follow the standard Landau asymptotic notation (see \cite{TaoVu} for instance): For functions $f,g$, we write $f(n)=O(g(n))$ if there exists an absolute constant $C>0$ and an integer $n_0$ such that for all $n\ge n_0, |f(n)|\le C|g(n)|$. We write $f=\Theta(g)$ if $f=O(g)$ and $g=O(f)$. If the constant $C=C(k)$ depends on another parameter $k$ (but not on $n$) then we shall denote this by writing $f=O_k(g)$. We also write $f\ll g$ if $\lim_{n\to\infty}\frac{f(n)}{g(n)}=0$. 
 
Suppose $G$ is a finite abelian group (written additively) with $\exp(G)=n$, and suppose $A\subseteq[1,n-1]$. The {\it $A$-weighted Davenport constant} of the group $G$ (introduced in Adhikari {\it et al}, see \cite{ACFKP}) is the least positive integer $k$ for which the following holds: Given an arbitrary sequence $(x_1,\ldots,x_k)$, with $x_i\in G$, there exists a non-empty subsequence $(x_{i_1},\ldots,x_{i_t})$ along with $a_{j}\in A$ such that $\displaystyle\sum_{j=1}^t a_jx_{i_j}=0$, where as usual, $ax=\overbrace{x+\cdots+x}^{a\textrm{\ times}}$. The weighted Davenport constant has been the primary object of study in several papers (see \cite {ACFKP, AC} for instance and some of the references in \cite{BM}) and determining $D_A(G)$  for some `natural' choices for the weight set $A$ for various categories of groups are questions that have garnered sufficient interest. 

In \cite{BM}, Balachandran and Mazumdar introduced a natural extremal problem associated to the weighted Davenport constant which is as follows. Suppose $G$ is a finite abelian group with $\exp(G)=n$ and $k\ge 2$. Define 
\begin{eqnarray*}\label{fdG_def} \fD_G(k)&:=&\min\{|A|: \emptyset \neq A\subseteq[1,n-1]\textrm{\ satisfies\ }D_A(G)\le k\},\\ 
                                         &:=&\infty\textrm{\ 
                                         if\ there\ is\ no\ 
                                         such\ }A.\end{eqnarray*}
Given a group $G$, determine $\fD_G(k)$ (if it is finite).

As it turns out, $\fD_G(k)<\infty$ for $k$ not `too large' (see \cite{BM}) and the most interesting case is when $G=\bF_p$ with $p$ being a prime (We  denote $\fD_G(k)$ by $\fD(p,k)$ for simplicity). 

In addition to being an interesting problem of independent interest, the problem of determining $\fD(p,k)$ also appears to be a generalization of certain well-studied notions for abelian groups. For instance, determining $\fD(p,2)$ is equivalent to finding a smallest difference base in the cyclic group $\Z_p$ (see \cite{BanGav} for a definition and related results)  and the general case is a vast generalization of this notion.

One of the main results in \cite{BM} is the following 
\begin{thm}\label{extreme_prob} Let $k\in\bN$, $k\ge 2$.  
\begin{enumerate}
\item[(a)] There exists an integer $p_0(k)$ and an absolute constant $C=C(k)>0$ such that for all prime $p>p_0(k)$, $$ p^{1/k}-1\le \fD(p,k)\le C(p\log p)^{1/k}.$$
\item[(b)] $\fD(p,k)\le Cp^{1/k}$ for $k=2,4$, for some absolute constant $C>0$.\end{enumerate}\end{thm}

It was conjectured in \cite{BM} that $\fD(p,k)=O(p^{1/k})$. The main result of this short paper is the following 
\begin{thm} \label{optfD} Suppose $k\ge 2$. Then there exists $p_0=p_0(k)$ such that for all primes $p\ge p_0$, 
$$\fD(p,k)\le 4^{k^2}p^{1/k}$$ so, in particular, $\fD(p,k)=\Theta_k(p^{1/k})$. \end{thm}

This does not settle the aforementioned conjecture in the strong form mentioned there, but it is a substantial improvement on theorem \ref{extreme_prob}. The constants that are involved in our proof are far from optimal, and we make no attempt to optimize for them. 

We prove theorem \ref{optfD} in the next section. We  ignore floors and ceilings in the expressions that appear, in order to increase clarity of expression and facilitate ease of comprehension. The last section contains a few remarks and questions for further inquiry.
\section{Proof of theorem \ref{optfD}}
We first consider the case $\fD(p,2k)$. We generalize the proof of the case $k=4$ in \cite{BM} along with some other ideas. The basic scheme of proof is somewhat similar, so we recall it first, for the readers' convenience.

 In order to show $\fD(p,2k)\le O_k(p^{1/2k})$, it suffices to show the existence of $A\subset \bF_q^*$ of size $27^{k^2}p^{1/2k}$ (which is stronger than the result stated) such that for any $\alpha_1,\ldots,\alpha_{k-1},\beta_1,\ldots,\beta_{k-1}\in\bF_p^*$,
\begin{eqnarray}\label{condition_2k}\bF_p^*\subseteq 
\frac{A+\alpha_1 A+\cdots+\alpha_{k-1}A}
{A+\beta_1 A+\cdots+\beta_{k-1}A}.\end{eqnarray}

To see why this will suffice, note that $D_A(\bF_p)\le 2k$ implies that for any sequence $(x_1,\ldots,x_{2k})\in(\bF_p^*)^{2k}$,  we have $0\in Ax_1+\cdots+Ax_{2k}$. This is equivalent to requiring that
$$-\frac{x_1}{x_{k+1}} = \frac{a_{k+1}+a_{k+2}(x_{k+2}/x_{k+1})+\cdots+a_{2k}(x_{2k}/x_{k+1})}{a_1+a_2(x_2/x_1)+\cdots+a_k(x_k/x_1)}$$ holds for some $a_i\in A$,  
and if (\ref{condition_2k}) holds, then this is indeed satisfied. 

The following observation, which is also the starting point in \cite{BM}, is again the key to our scheme of proof.

\begin{obs}\label{obs0} If  $A,B\subseteq\bF_p$ satisfy $|A||B|>p$, then $\bF_p=\frac{A-A}{B-B}$.\end{obs} 

Hence to show (\ref{condition_2k}), it suffices to show the existence of a set $A$ of size at most the bound mentioned earlier, that satisfies the following:  $$\textrm{For\  any\ } \alpha_1\ldots,\alpha_{k-1}\in\bF_p^*,  \ \ |A+\alpha_1A+\cdots+\alpha_{k-1}A|>\sqrt{p}.$$
Let $L=Cp^{1/2k}$ where we shall determine $C$ later.  For a positive integer $t$, let $X_t:=t[-L,L]=\{-Lt,\ldots,-t, 0, t,\ldots,Lt\}$. It immediately follows that  $\alpha X_t=X_{\alpha t}$.  The following observation also follows from a  simple inductive argument (also see \cite{TaoVu} for generalized arithmetic progressions in finite abelian groups).

\begin{obs}
\label{obs1} For distinct $t_i$, 
$X_{t_1}+\cdots+ X_{t_{k}}$ contains a subset $Y$ 
of size at least $|X_{t_1}+X_{t_2}+\cdots+ X_{t_{k}}|/2^k$ 
such that $Y-Y\subseteq X_{t_1}+X_{t_2}+\cdots+ X_{t_{k}}$.\end{obs}

We shall now introduce some notation.
We shall write $I:=X_1=[-L,\ldots,-1,0,1,\ldots,L]$ for convenience. For sets $A_1,\ldots, A_r\subset\bF_p$ we shall denote the sum set $A_1+\cdots+A_r$ by $\sum_{i=1}^r A_i$. For  $t\in\bF_p^*$, we shall denote the set $\{ta:a\in A\}$ by $tA$.  We write $A_*:=A\setminus\{0\}$, and  finally, for $\mathbf{a}:=(\alpha_1,\ldots,\alpha_k),\mathbf{b}:=(\beta_1,\ldots,\beta_k)$ we shall write $\mathbf{a}\cdot\mathbf{b}:=(\alpha_1\beta_1,\ldots,\alpha_k\beta_k)$.


   
  For $A_1,\ldots, A_{k-1}\subseteq\bF_p$ define
 $$S(A_1,\ldots,A_{k-1}; B):= \left\{(t_1,\ldots,t_{k-1}) \in (\bF_p^*)^{k-1} :  (\sum_{i=1}^{k-1}t_iA_i) \cap B\neq\emptyset\right\}.$$
 
 \begin{defn} For a given $\bft:=(t_1,\ldots,t_{k-1}) \in (\bF_p^*)^{k-1}$, we say that  $\mathbf{a}:=(\alpha_1,\ldots,\alpha_{k-1}) \in (\bF_p^*)^{k-1}$ {\it is good for} $\mathbf{t}$ if 
 $$\left|(I+\alpha_1 X_{t_1}+\cdots+ \alpha_{k-1}X_{k-1})_*\right|\ge  \frac{L^k}{4^k}.$$\end{defn}
 
 Suppose $\bft\in(\bF_p^*)^{k-1}$. The following lemma tells us that if $\bfa$ is not good for $\bft$ then, in a sense (made precise by the lemma), $\bfa$ is restricted to a very small subset of $(\bF_p^*)^{k-1}$.  
\begin{lem}\label{small} Let $L$ as before, $B := [-2L, 2L]$, and for each $1\le i,j\le k-1$, define $A_i(j):= [\frac{-L}{2},\frac{L}{2}]$ if $i\ne j$, and $A_i(i)=[1,L/2]$. Let
 $\mcS_j:=S(A_1(j),\ldots,A_{k-1}(j);B)$ and $\mcS:=\displaystyle\cup_{j=1}^{k-1} \mcS_j$. Suppose $\bfa:=(\alpha_1,\ldots,\alpha_{k-1}) \in (\bF_p^*)^{k-1}$ is not good for $\bft:=(t_1,\ldots,t_{k-1})\in(\bF_p^*)^{k-1}$,  then $\mathbf{a}\cdot\mathbf{t} \in \mcS$. Moreover, $$|\mcS| < 3kL^{k}(p-1)^{k-2}< 3kC^k(p-1)^{k-\frac{3}{2}}\textrm{\ for\ sufficiently\ large\ }p.$$
\end{lem} 
\begin{proof} Since 
$\alpha X_{t}=X_{\alpha t}$, it suffices to show that 
  $$|(I+ X_{t_1}+\cdots+ X_{t_{k-1}})_*|< \frac{L^k}{4^k}\ \Rightarrow\ \bft\in \mcS.$$

Set $X_t^+:=\{0,t,\ldots,Lt\}$. If $\mfj$ denotes the $(k-1)$-tuple $(j_1,\ldots,j_{k-1})$, then observe that  
$$I+X_{t_1}^++\cdots+ X_{t_{k-1}}^+=
\displaystyle\bigcup_{\mfj\in[0,L]^{k-1}} 
[j_1t_1+\cdots+j_{k-1}t_{k-1}-L,j_1t_1+\cdots+j_{k-1}t_{k-1}+L].$$

Put an arbitrary linear order $\le$ on  $[0,L]^{k-1}$ with least element ${\bf 0}:=(0,\ldots,0)$ and define \begin{eqnarray*}
\X({\bf 0})&:=&[-L,L],\\
 \X(\mfi)&:=&\bigcup_{\mfj\le\mfi}\ [\mfj\cdot\bft - L,\mfj\cdot\bft +L] =\bigcup_{\mfj\le\mfi}\  [j_1t_1+\cdots+j_{k-1}t_{k-1}-L, j_1t_1+\cdots+j_{k-1}t_{k-1}+L].\end{eqnarray*}

Call the set $\X(\mfi)$  {\it valid} if it is the union of pairwise disjoint intervals each of length $2L$, centred around an element 
of $\displaystyle\sum_{j=1}^{k-1}X_{t_j}^+$. Clearly, $\X(0,\ldots,0)$ is valid.

 We now claim that there exists $\mfi$ with some $i_j\le L/2$ such that $\X(\mfi)$ is not valid. Indeed, suppose $\X(\mfi)$ is valid for all such $\mfi$. In particular, for $M_j=\lceil L/4\rceil$, 
we have $$|\X(M_1,\ldots,M_{k-1})|= 2L\prod_{j=1}^{k-1}M_j\ge (2L)\left(\frac{L}{4}\right)^{k-1}\ge 2\frac{L^k}{4^{k-1}}$$ 
contradicting the hypothesis. 

Let $\mfi=({i_1},\ldots,{i_{k-1}})$ be the first $(k-1)$-tuple with respect to the linear order for which $\X({i_1},\ldots,{i_{k-1}})$ is not valid. In particular, there exists $\mfj=(j_1\ldots,j_{k-1})\ne\mfi$ and $1\le r\le k-1$ such that $i_r<j_r$ and
$$\sum_{l=1}^{k-1}i_lt_l + \xi_1 = \sum_{l=1}^{k-1}j_lt_l + \xi_2$$ for some $\xi_1,\xi_2\in [-L,L]$. Consequently, 
$$\bft\cdot(\mfj -\mfi)=\xi_1-\xi_2\ \Rightarrow\ \bft\in\mcS_r.$$

To complete the proof, we need to show the bound on $|\mcS|$. We shall first show that $|\mcS_1|< 3L^k(p-1)^{k-2}$ for sufficiently large $p$.

First observe that $(t_1\ldots,t_{k-1})\in \mcS_1$ implies that there exists $a_i\in A_i(1)$ for $1\le i\le k-1$ and $b\in B$ such that $t_1a_1+\cdots t_{k-1}a_{k-1}=b$. 
For fixed choices of $a_i\in A_i(1)$ for $1\le i\le k-1$ and $b\in B$, along with choices of $t_2,\ldots,t_{k-1}\in\bF_p^*$, the equation $t_1a_1+\cdots t_{k-1}a_{k-1}=b$ admits a unique solution for $t\in\bF_p$. In particular, it follows that $$|\mcS_1|\le (L/2)\cdot (L+1)^{k-2}\cdot(4L+1)\cdot (p-1)^{k-2} <3L^kp^{k-2}$$ for $p$ sufficiently large as claimed. 

Now the bound on $|\mcS|$ follows by a similar bound for each $|\mcS_j|$.
\end{proof}



 In the rest of the paper, $\mcS$ shall denote the set described in the statement of lemma \ref{small}.
\begin{lem}\label{intersect} Suppose that there exist $\bfy_i=(y_1^{(i)},\ldots, y_{k-1}^{(i)})\in (\bF_p^*)^{k-1}$ for $1\le i\le N$ such that 
$$\displaystyle\bigcap_{i=1}^N \bfy_i\cdot \mcS =\emptyset.$$ Write $x_r^{(i)}=(y_r^{(i)})^{-1}$ for $1\le i\le N, 1\le r\le k-1$, and set $\bfx_i=(x_1^{(i)},\ldots,x_{k-1}^{(i)})$.
Then the set 
$$A=\displaystyle \left(I\cup\bigcup_{\substack{1\le j\le N\\1\le r\le k-1}}X_{x_j^{(r)}}\right)_*\subset [1,p-1]$$ 
satisfies $D_A(\bF_p)\le 2k$. In particular, $\fD(p,2k)\le 2kNL$.\end{lem}
\begin{proof} Let $\bfa = (\alpha_1,\ldots,\alpha_{k-1})\in (\bF_p^*)^{k-1}$. We shall show that
$$|I+\alpha_1 X_{x_1^{(i)}}+\cdots+ \alpha_{k-1} X_{x_{k-1}^{(i)}}| \ge \frac{L^k}{4^k}$$
for some $1\le i\le N$. Then by observation \ref{obs1}, 
there exists $Y_{\bfa}\subset I+\alpha_1 X_{x_1^{(i)}}+\ldots+ \alpha_{k-1} X_{x_{k-1}^{(i)}}$
with $|Y_{\bfa}|>p^{1/2}$ (if $C>8$) such that 
$Y_{\bfa}-Y_{\bfa}\subseteq I+\alpha_1 X_{x_1^{(i)}}+\ldots+ \alpha_{k-1} X_{x_{k-1}^{(i)}}$. 
Since this holds for all $\bfa\in(\bF_p^*)^{k-1}$ the proof of lemma \ref{intersect} is complete by observation \ref{obs0}.

Since $\displaystyle\bigcap_0^{N} \bfy_i\cdot\mcS=\emptyset$, there exists $1\le i\le N$ such that  $\bfa\not\in \bfy_i\cdot \mcS$. 
But then, by lemma \ref{small}, this implies that  $\bfa$ is good for  $\bfy_i^{-1}=\bfx_i$, or equivalently,
$$|I+\alpha_1 X_{x_1^{(i)}}+\ldots+ \alpha_{k-1} X_{x_{k-1}^{(i)}}| \ge \frac{L^k}{4^k}$$ as required. \end{proof}
We are now in a position to prove theorem \ref{optfD}.
\begin{proof} (of theorem \ref{optfD}) We shall denote by $\bfy$, a typical element in $(\bF_p^*)^{k-1}$. For $\Y = (\bfy_1,\ldots,\bfy_{2k-3})\in((\bF_p^*)^{k-1})^{2k-3}$, we say that $\mfa = (\bfa_0,\bfa_1,\ldots,\bfa_{2k-3})\in \mcS^{2k-2}$ is {\it binding} for $\Y$ if
$$\bfa_0 =\bfy_1\cdot\bfa_1= \cdots = \bfy_{2k-3}\cdot \bfa_{2k-3}.$$
We shall call $\bfa_0$ the leading element of $\mfa$. Clearly, each $\mfa$ determines a unique $\Y\in((\bF_p^*)^{k-1})^{2k-3}$ such that $\mfa$ is binding for $\Y$. For $\Y\in((\bF_p^*)^{k-1})^{2k-3}$ define
\begin{eqnarray*}
\mfa(\Y)&:=&\{\mfa:\mfa \text{ is binding for }\Y\},\\
 N(\Y)&:=& |\mfa(\Y)|,\\
\textrm{NORMAL}&:=& \{\Y \in ((\bF_p^*)^{k-1})^{2k-3} : N(\Y) \leq 2(3kC^k)^{2k-2}\}.\end{eqnarray*}

Suppose $\Y$ is chosen uniformly at random from $((\bF_p^*)^{k-1})^{2k-3}.$  Then
$$\mathbb{E}(N(\Y)) = \sum_{\mfa\in \mcS^{2k-2}} \mathbb{P}\left(\mfa \text{ is binding for  } \Y\right) =\frac{|\mcS|^{2k-2}}{(p-1)^{(k-1)(2k-3)}}<(3kC^k)^{2k-2}=C^*, \textrm{\ (say)}$$ 
Hence by Markov's Inequality, it follows that 
\begin{eqnarray}\label{NRMLD}
|\text{NORMAL}|\geq \frac{1}{2}(p-1)^{(k-1)(2k-3)}.
\end{eqnarray}

 Pick $\Y_1=(\bfy_1^{(1)},\ldots,\bfy_{2k-3}^{(1)}) \in$ NORMAL arbitrarily. Write $\mfa(\Y_1)= \{\mfa_1,\ldots,\mfa_{\ell}\}$ with $\ell\le C^*$. Let $\mfa_i=(\bfa_0[i],\bfa_1[i],\ldots,\bfa_{2k-3}[i])$. 
 Since $\Y_1\in\textrm{NORMAL}$, 
$$\bfa_0[i]=\bfy_1^{(1)}\cdot\bfa_1[i] =\cdots=\bfy_{2k-3}^{(1)}\cdot\bfa_{2k-3}[i]\textrm{\ \ for\ all\ }1\le i\le \ell.$$  
The number of  $\Y\in\textrm{NORMAL}$ such that $(\bfa_0[1],\bfb_1,\ldots,\bfb_{2k-3})\in\mfa(\Y)$ is at most $|S|^{2k-3}\ll |\textrm{NORMAL}|$, so that in particular, there exists $\Y_2=(\bfy_1^{(2)},\ldots,\bfy_{2k-3}^{(2)})\in\textrm{NORMAL}$ such that $\bfa_0[1]$ is not a leading element of any $\mfa\in\mfa(\Y_2)$. Thus, having chosen $\Y_1,\ldots,\Y_i\in\textrm{NORMAL}$, we inductively  pick $\Y_{i+1}\in\textrm{NORMAL}$ such that $\bfa_0[i]$ is not a leading element of any $\mfa\in\mfa(\Y_{i+1})$, and these choices are possible by the same argument described above, and note that this procedure terminates in at most $T\le \ell$ steps. As a consequence of these choices, it follows immediately that 
$$\mcS\cap\left(\bigcap_{\substack{1\le i\le 2k-3\\ 1\le j\le T}} \bfy_i^{(j)}\cdot\mcS \right)=\emptyset$$
so we make take $N=1+(2k-3)T\le 2k (3kC^k)^{2k-2}$.

Putting all the ingredients together, we have $$\fD(p,2k)\le 2kNL\le (4k^2) (3kC^k)^{2k-2} (Cp^{1/2k})\le (27)^{k^2}p^{1/2k}$$ for all $k\ge 2$.

For the odd case, the proof moves along exactly the same lines. To bound $\fD(p, 2k+1)$, we need to describe a set $A\subset\bF_p^*$ such that for any $\alpha_1,\ldots,\alpha_{k},\beta_,\ldots,\beta_{k-1}\in\bF_p^*$ 
$$\bF_p^*\subseteq\frac{A+\alpha_1A+\cdots+\alpha_k A}{A+\beta_1A+\cdots+\beta_{k-1}A}$$ holds. Keeping with our scheme of proof, it will suffice to show that for $r=k-1,k$, and any $\alpha_1\ldots,\alpha_{r}\in\bF_p^*$ there exists $A\subset\bF_p^*$ such that 
$|A+\alpha_1A+\cdots+\alpha_rA|>p^{\frac{r+1}{2k+1}}$. We now imitate the same argument to obtain $A_r\subset\bF_p^*$ that works for $r$, and finally $A=A_{k-1}\cup A_k$ does the job. It is  somewhat routine to check that the bound on $|A|$ as stated in theorem \ref{optfD} indeed holds, so we omit those details.
\end{proof}

\section{Concluding remarks}
\begin{itemize}
\item It should be quite clear that the dependence on $k$ in the constant in theorem \ref{optfD} is far from best possible. We still believe that $\fD(p,k)\le Cp^{1/k}$ for an absolute constant $C$, which is still out of our reach.
\item A closer inspection of the proof of theorem \ref{extreme_prob} in \cite{BM} actually reveals that the proof of the first part of theorem \ref{extreme_prob} holds for all $k\ll\frac{\log p}{\log\log p}$, which makes that theorem quite robust because the problem of determining $\fD(p,k)$ is relevant only for $k\le\log_2 p+1$ (see proposition 4.1 in \cite{BM}). In contrast, the bound in theorem \ref{optfD} is suboptimal to the bound in theorem \ref{extreme_prob} unless $k\ll (\log\log p)^{\frac{1}{3}}$. 
\item As pointed out in \cite{BM} the problem of determining $\fD_G(k)$ for arbitrary abelian groups $G$ reduces to the case(s) $G=\bF_p$ (resp. $G=\bF_p^r$)  as the most relevant one because one can choose weights that project all the weighted elements in to some small subgroup of $G$. This raises a more interesting {\it irreducible}  variant of the same extremal problem: Given a finite abelian group $G$ with $\exp(G)=n$, define  
\begin{eqnarray*}\label{fdG_def} I\fD_G(k)&:=&\min\{|A|: \emptyset \neq A\subseteq\Z_n^*\textrm{\ satisfies\ }D_A(G)\le k\},\\ 
                                         &:=&\infty\textrm{\ 
                                         if\ there\ is\ no\ 
                                         such\ }A.\end{eqnarray*}
Given a group $G$, determine $I\fD_G(k)$ (if it is finite). 

This latter extremal function does not permit us the trick of projecting into a smaller subgroup, and henceforth poses more interesting possibilities.
\end{itemize}

\end{document}